\documentclass[10pt]{article}
\font\smallit=cmti10
\font\smalltt=cmtt10

\usepackage{amssymb,latexsym,amsmath,epsfig,amsthm} 
\usepackage{empheq}
\usepackage{ascmac}
\makeatletter

\renewcommand\section{\@startsection {section}{1}{\z@}
{-30pt \@plus -1ex \@minus -.2ex}
{2.3ex \@plus.2ex}
{\normalfont\normalsize\bfseries\boldmath}}

\renewcommand\subsection{\@startsection{subsection}{2}{\z@}
{-3.25ex\@plus -1ex \@minus -.2ex}
{1.5ex \@plus .2ex}
{\normalfont\normalsize\bfseries\boldmath}}

\renewcommand{\@seccntformat}[1]{\csname the#1\endcsname. }

\makeatother
\newtheorem{theorem}{Theorem}
\newtheorem{lemma}{Lemma}

\theoremstyle{definition}
\newtheorem{defn}{Definition}[section]
\newtheorem{rem}{Remark}[section]

\newtheorem{pict}{Figure}[section]

\begin{document}

\begin{center}
\uppercase{\bf Restricted Nim with a Pass}
\vskip 20pt
{\bf Ryohei Miyadera }\\
{\smallit Keimei Gakuin Junior and High School, Kobe City, Japan}. \\
{\tt runnerskg@gmail.com}
\vskip 10pt
{\bf Hikaru Manabe}. \\
{\smallit Keimei Gakuin Junior and High School, Kobe City, Japan}. \\
{\tt urakihebanam@gmail.com}

\vskip 10pt

\end{center}
\vskip 20pt
\centerline{\smallit Received: , Revised: , Accepted: , Published: } 
\vskip 30pt

\pagestyle{myheadings} 
\markright{\smalltt INTEGERS: 19 (2019)\hfill} 
\thispagestyle{empty} 
\baselineskip=12.875pt 
\vskip 30pt

\centerline{\bf Abstract}
\noindent
This paper presents a study of restricted Nim with a pass. In the restricted Nim considered in this study, two players take turns and remove stones from the piles. In each turn, when the number of stones is $m$, each player is allowed to remove at least one stone and at most $\lceil \frac{m}{2}  \rceil$ stones from a pile of $m$ stones. The standard rules of the game are modified to allow a one-time pass, that is, a pass move that may be used at most once in the game and not from a terminal position. Once a pass has been used by either player, it is no longer available. It is well-known that in classical Nim, the introduction of the pass alters the underlying structure of the game, significantly increasing its complexity.
In the restricted Nim considered in this study, the pass move was found to have a minimal impact. There is a simple relationship between the Grundy numbers of restricted Nim and the Grundy numbers of restricted Nim with a pass, where the number of piles can be any natural number. Therefore, the authors address a longstanding open question in combinatorial game theory:
the extent to which the introduction of a pass into a game affects its behavior. The game that we developed appears to be the first variant of Nim that is fully solvable when a pass is not allowed and remains fully solvable following the introduction of a pass move.
 \pagestyle{myheadings} 
 \markright{\smalltt \hfill} 
 \thispagestyle{empty} 
 \baselineskip=12.875pt 
 \vskip 30pt
 
\section{Introduction}
In this study, restricted Nim and restricted Nim with a pass are examined. An interesting but difficult question in combinatorial game theory has been to determine what happens when standard game rules are modified to allow a one-time pass, that is, a pass move that may be used at most once in the game and not from a terminal position. Once a pass has been used by either player, it is no longer available. In the case of classical Nim, the introduction of the pass alters the mathematical structure of the game, considerably increasing its complexity. The effect of a pass on classical Nim remains an important open question that has defied traditional approaches. The late mathematician David Gale offered a monetary prize to the first person to develop a solution for three-pile classical Nim with a pass.

In \cite{nimpass} (p. 370), Friedman and Landsberg conjectured that ”solvable combinatorial games are structurally unstable to perturbations, while generic, complex games will be structurally stable.” One way to introduce such a perturbation is to allow a pass. 
One of the authors of the present article reported a counterexample to this conjecture in \cite{integers1}. The game used in \cite{integers1} is solvable because there is a simple formula for the Grundy numbers, and even when we introduce a pass move to the game, there is a simple formula for $\mathcal{P}$-positions.

The restricted Nim considered in the present study is of the same type, but the introduction of a pass move has a minimal impact. There is a simple relationship between the Grundy numbers of the game and the Grundy numbers of the game with a pass move, and the number of piles can be any natural number. This result is stated in Theorem \ref{grundytwopilepass} of the present article.
One of the authors discussed part of the result for one-pile restricted Nim with a pass in \cite{jcdcg2018a}.

Let $Z_{\ge 0}$ and $N$ be sets of non-negative numbers and natural numbers, respectively.
For completeness, we briefly review some of the necessary concepts of combinatorial game theory. Details are presented in $\cite{lesson}$ and $\cite{combysiegel}$. 
\begin{defn}\label{definitionfonimsum11}
	Let $x$ and $y$ be non-negative integers. They are expressed in Base 2 as follows: 
$x = \sum_{i=0}^n x_i 2^i$ and $y = \sum_{i=0}^n y_i 2^i$, with $x_i,y_i \in \{0,1\}$.
	We define \textit{nim-sum} $x \oplus y$ as follows:
\begin{equation}
	x \oplus y = \sum\limits_{i = 0}^n {{w_i}} {2^i},
\end{equation}
	where $w_{i}=x_{i}+y_{i} \ (\bmod\ 2)$.
\end{defn}

For impartial games without drawings, there are only two outcome classes.
\begin{defn}\label{NPpositions}
	$(a)$ A position is referred to as a $\mathcal{P}$-\textit{position} if it is a winning position for the previous player (the player who has just moved), as long as he/she plays correctly at every stage. \\
	$(b)$ A position is referred to as an $\mathcal{N}$-\textit{position} if it is a winning position for the next player as long as he/she plays correctly at every stage.
\end{defn}

\begin{defn}\label{defofmexgrundy}
$(i)$ For any position $\mathbf{p}$ of game $\mathbf{G}$, there is a set of positions that can be reached by precisely one move in $\mathbf{G}$, which we denote as \textit{move}$(\mathbf{p})$. \\	
$(ii)$ The \textit{minimum excluded value} $(\textit{mex})$ of a set $S$ of non-negative integers is the smallest non-negative integer that is not in S. \\
$(iii)$ Let $\mathbf{p}$ be the position of an impartial game. The associated Grundy number is denoted as $G(\mathbf{p})$ and is recursively defined as follows:
$\mathcal{G}(\mathbf{p}) = \textit{mex}(\{\mathcal{G}(\mathbf{h}): \mathbf{h} \in move(\mathbf{p})\}).$
\end{defn}
\begin{defn}\label{sumofgames}
	The \textit{disjunctive sum} of the two games, which is denoted as $\mathbf{G}+\mathbf{H}$, is a supergame in which a player may move in either $\mathbf{G}$ or $\mathbf{H}$ but not both.
\end{defn}

\begin{theorem}\label{thofsumofgame}
Let $\mathbf{G}$ and $\mathbf{H}$ be impartial rulesets and $G_{\mathbf{G}}$ and $G_{\mathbf{H}}$ be the Grundy numbers of position $\mathbf{g}$ played under the rules of $\mathbf{G}$ and position $\mathbf{h}$ played under the rules of $\mathbf{H}$, respectively. Thus, we have the following:\\
	$(i)$ For any position $\mathbf{g}$ of $\mathbf{G}$, 
	$G_{\mathbf{G}}(\mathbf{g})=0$ if and only if $\mathbf{g}$ is the $\mathcal{P}$ position. \\
	$(ii)$ The Grundy number of positions $\{\mathbf{g},\mathbf{h}\}$ in game $\mathbf{G}+\mathbf{H}$ is	$G_{\mathbf{G}}(\mathbf{g})\oplus G_{\mathbf{H}}(\mathbf{h})$.
\end{theorem}

For the proof of this theorem, see $\cite{lesson}$.
\begin{rem}
With  Theorem \ref{thofsumofgame}, we can find a $\mathcal{P}$-position by calculating the Grundy numbers and a $\mathcal{P}$-position of the sum of two games by calculating the Grundy numbers of two games.
Therefore, Grundy numbers are an important research topic in combinatorial game theory.
\end{rem}	
	
\section{Maximum Nim}
In this section, we study maximum Nim, which is a game of restricted Nim.
\begin{defn}\label{defofregular}
 If the sequence $f(m)$ for $m \in Z_{\ge 0}$ satisfies
$0 \leq f(m) -f(m-1) \leq 1$ for any natural number $m$, it is called a “regular sequence.”
\end{defn}

\begin{defn}\label{defofmaxnim}
Let $f(m)$ be a regular sequence:
Suppose that there is a pile of $n$ stones, and two players take turns removing stones from the pile.
In each turn, the player is allowed to remove at least one stone and at most $f(m)$ stones, where $m$ represents the number of stones. The player who removes the last stone is the winner. 
 We refer to $f$ as a “rule sequence.”
\end{defn}

\begin{lemma}\label{lemmabylevinenim} Let $\mathcal{G}$ represent the Grundy number of the maximum Nim with the rule sequence $f(x)$. Then, we have the following properties: \\
$(i)$ If $f(x) = f(x-1)$, $\mathcal{G}(x) = \mathcal{G}(x-f(x)-1)$.\\
$(ii)$ If $f(x) > f(x-1)$, $\mathcal{G}(x) = f(x)$.
\end{lemma}
\begin{proof}
Properties $(i)$ and $(ii)$ are proven in Lemma 2.1 of \cite{levinenim}.
\end{proof}

\subsection{Maximum Nim Whose Rule Sequence is $f(x) = \lceil \frac{x}{2}  \rceil$}\label{sectionforceil}
In this section, we let $f(x) = \lceil \frac{x}{2}  \rceil$. Because $0 \leq f(m) -f(m-1) \leq 1$ for any $m \in N$, $f(m)$ for $m \in Z_{\ge 0}$ is a regular sequence.
Here, we examine the maximum Nim of Definition \ref{defofmaxnim} for $f(x)$.

Another option is to use $f(x) = \lfloor \frac{x}{2}\rfloor $; however, this case produces almost the same result because $ \lfloor \frac{x+1}{2}\rfloor  = \lceil \frac{x}{2}  \rceil$ for any $n \in Z_{\ge 0}$. Therefore, the case of $f(x) = \lfloor \frac{x}{2}\rfloor $ is omitted in this study.

\begin{defn}
We denote the pile of $m$ stones as $(m)$, which we call the position of the game.
\end{defn}

We define $\textit{move}(t)$ for the maximum Nim of Definition \ref{defofmaxnim} for the rule sequence $f(x)$.

\begin{defn}\label{moveofvpsnim}
$\textit{move}(t)$ is the set of all the positions that can be reached from position $(t)$.
 For any $t \in Z_{\ge 0}$, we have 
	\begin{flalign}
	 & \textit{move}(t)	= \{(t-v):v \leq \lceil \frac{t}{2}  \rceil \text{ and } v \in N \}. & \nonumber 
	\end{flalign}
\end{defn}


\begin{lemma}\label{grundyhalfnim} 
Let $\mathcal{G}$ represent the Grundy number of the maximum Nim with the rule sequence $f(x) = \lceil \frac{x}{2}  \rceil$. Then, we have the following properties: \\
$(i)$	If $t$ is even and $t \geq 2$, $\mathcal{G}(t) = \mathcal{G}(\frac{t-2}{2})$. \\
$(ii)$ 	If $t$ is odd, $\mathcal{G}(t) =\frac{t+1}{2}$.
\end{lemma}
\begin{proof}
\textbf{(i)}	If $t$ is even, $\lceil \frac{t}{2}  \rceil = \lceil \frac{t-1}{2} \rceil$. Therefore, according to $(i)$ in Lemma \ref{lemmabylevinenim}, $\mathcal{G}(t) = \mathcal{G}(t-\lceil \frac{t}{2}  \rceil -1)$ $=\mathcal{G}(\frac{t-2}{2})$. \\
\textbf{(ii)} 	If $t$ is odd, $\lceil \frac{t}{2}  \rceil > \lceil \frac{t-1}{2} \rceil$. Therefore, according to $(ii)$ of Lemma \ref{lemmabylevinenim}, we have $\mathcal{G}(t) = \lceil \frac{t}{2} \rceil$ $= \frac{t+1}{2}$. \\
\end{proof}

\subsection{Three-Pile Maximum Nim}\label{twopilepass}

\begin{defn}\label{defofmaxnim3piles}
Suppose that there are three piles of stones and two players take turns to remove stones from the piles.
In each turn, the player chooses a pile and removes at least one stone and at most $f(x) = \lceil \frac{x}{2}  \rceil$ stones, where $x$ represents the number of stones. The player who removes the last stone is the winner. 
 The position of the game is represented by three coordinates $\{s,t,u\}$, where $s$, $t$, and $u$ represent the numbers of stones in the first, second, and third piles, respectively.
\end{defn}

According to the results presented in Section \ref{sectionforceil} and Theorem \ref{thofsumofgame}, we can calculate the Grundy numbers of the game in Definition \ref{defofmaxnim3piles}.

\begin{theorem}\label{thmforsumgame}
Let $\mathcal{G}(t)$ be the Grundy number of the game in Subsection \ref{sectionforceil}.
Then, the Grundy number $\mathcal{G}(s,t,u)$ of the game of Definition \ref{defofmaxnim3piles} satisfies the following equation:
$\mathcal{G}(s,t,u) = \mathcal{G}(s) \oplus \mathcal{G}(t)\oplus \mathcal{G}(u)$.
\end{theorem}
\begin{proof}
This is directly from Theorem \ref{thofsumofgame}.
\end{proof}

\section{Maximum Nim with a Pass}\label{pass1}
In Subsections \ref{sectionforceilpass} and \ref{twowithpass}, we modify the standard rules of the games to allow for a one-time pass, that is, a pass move that may be used at most once in the game and not from a terminal position. Once a pass has been used by either player, it is no longer available.

\subsection{Maximum Nim with a Pass Whose Rule Sequence is  $f(x) = \lceil \frac{x}{2}  \rceil$ with a pass move}\label{sectionforceilpass}
The position of this game is represented by two coordinates $\{t, p\}$, where $t$ represents the number of stones in the pile. $p = 1$ if the pass is still available; otherwise, $p = 0$.

We define $\textit{move}$ in this game.
\begin{defn}\label{defofonenimepass}
 For any $t \in Z_{\ge 0}$, we have $(i)$ and $(ii)$. \\
 $(i)$ If $p=1$ and $t>0$,
\begin{equation}
 \textit{move}(t,p)= \{\{t-v,p\}:v \leq \lceil \frac{t}{2}  \rceil \text{ and } v \in N \} \cup \{\{t,0\}\}.\nonumber 
\end{equation}
 $(ii)$ If $p=0$ or $t=0$,
\begin{equation}
\textit{move}(t,p)= \{\{t-v,p\}:u \leq \lceil \frac{t}{2}  \rceil \text{ and } v \in N \}. \nonumber 
\end{equation}
\end{defn}
\begin{rem}
Note that a pass is not available from position $\{t,1\}$ with $t=0$, which is the terminal position.
It is clear that $\mathcal{G}(t,0)$ is identical to $\mathcal{G}(t)$ in Section \ref{sectionforceil}.
\end{rem}

According to Definitions \ref{defofmexgrundy} and \ref{defofonenimepass}, we define the Grundy number $\mathcal{G}(t,p)$ of the position $\{t,p\}$.

\includegraphics[height=1.5cm]{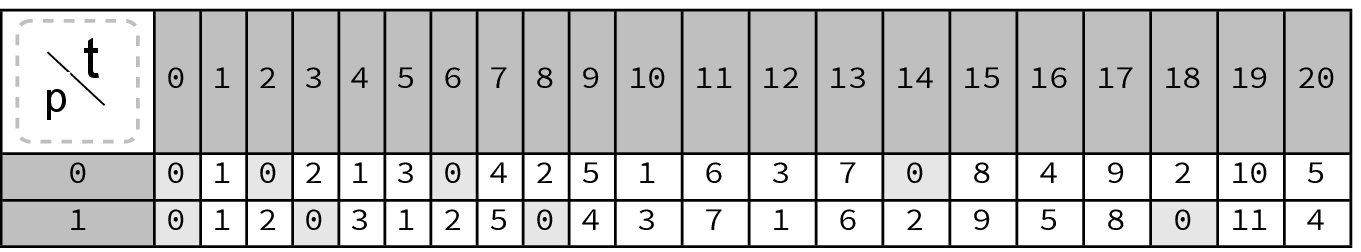}
\begin{pict}\label{Grundy number of half nim with a pass}
Table of Grundy numbers $\mathcal{G}(t,p)$
\end{pict}

	
\begin{theorem}\label{grundyhalfnimpass}
Let $\mathcal{G}(s,p)$ be the Grundy number of position $\{s,p\}$. Then, we obtain the following equations: \\
$(i)$  $\mathcal{G}(0,0) = 0$ and  $\mathcal{G}(0,1) = 0$. \\
$(ii)$   For $u \in N$,
if $\mathcal{G}(u,0)=0$, then        $ \mathcal{G}(u,1)=1.\nonumber$\\
$(iii)$	 For $u \in N$,
if $\mathcal{G}(u,0)=2$, then        $ \mathcal{G}(u,1)=0.\nonumber$\\
$(iv)$ 	For $u,m \in N$ such that $m>1$,
if $\mathcal{G}(u,0)=2m$, then        $ \mathcal{G}(u,1)=2m-1.\nonumber$\\
$(v)$	For $u,m \in N$, if $\mathcal{G}(u,0)=2m-1$, then        $ \mathcal{G}(u,1)=2m.$\nonumber
\end{theorem}
\begin{proof}
$(i)$ $\mathcal{G}(0,0)= 0$ because $\{0,0\}$ is the terminal position.
We cannot move to any position or use a pass move from position $\{0,1\}$.
Hence, $\{0,1\}$ is also a terminal position. Therefore,
$\mathcal{G}(0,1) = 0$.

Next, we prove $(ii)$, $(iii)$, $(iv)$, and $(v)$ using mathematical induction.
From $(ii)$ of Lemma \ref{grundyhalfnim}, 
\begin{equation}
 \mathcal{G}(1,0)=\mathcal{G}(1) = \frac{1+1}{2}= 1. \label{g10eq1}
\end{equation}
As $\textit{move}(1,1) =\{\{0,1\},\{1,0\}\}$,
$\mathcal{G}(0,1) = 0$, and $\mathcal{G}(1,0) = 1$, we have
\begin{align}
\mathcal{G}(1,1) & = \textit{mex}(\{\mathcal{G}(k,h):\{k,h\} \in \textit{move}(1,1)\}) \nonumber \\
& =\textit{mex}(\{\mathcal{G}(0,1),\mathcal{G}(1,0)\}), \nonumber \\
& =\textit{mex}(\{0,1\})=2. \label{g11eq2}
\end{align}
Let $t \in N$. From equations (\ref{g10eq1}) and (\ref{g11eq2}), we have only to prove the case such that
\begin{equation}
t \geq 2. \label{casegreat2}
\end{equation}
We suppose that $(ii)$, $(iii)$, $(iv)$, and $(v)$ are valid for $k \in Z_{\ge 0}$ such that $k < t $. From the inequality in (\ref{casegreat2}), we have 
\begin{equation}
\{0,0\} \notin \textit{move}(t,0)=\{\{t-1,0\}, \cdots, \{t-\lceil \frac{t}{2}  \rceil,0\}\}; \nonumber
\end{equation}
hence, we have
\begin{equation}
k \in N \label{kinn}
\end{equation}
when $\{k,0\} \in \textit{move}(t,0)$. \\
$(ii)$ Suppose that 
\begin{equation}
\mathcal{G}(t,0) = 0.\label{gtoeq0}
\end{equation}
\begin{equation}
\mathcal{G}(t,0) = \textit{mex}(\{\mathcal{G}(k,0):\{k,0\} \in \textit{move}(t,0)\}); \nonumber  
\end{equation}
hence, according to the definition of the Grundy number in Definition \ref{defofmexgrundy},  
\begin{align}
 0 \notin & \{\mathcal{G}(k,0):\{k,0\} \in \textit{move}(t,0)\} \nonumber \\
= & \{ \mathcal{G}(t-1,0), \cdots, \mathcal{G}(t-\lceil \frac{t}{2}  \rceil,0)\}. \nonumber
\end{align}
From relation (\ref{kinn}) and the mathematical induction hypothesis for $(ii)$, $(iii)$, $(iv)$, and $(v)$, 
\begin{equation}
1 \notin \{\mathcal{G}(t-1,1), \cdots, \mathcal{G}(t-\lceil \frac{t}{2}  \rceil,1)\}.\label{no1in}
\end{equation}
\begin{align}
 \mathcal{G}(t,1)& = \textit{mex}(\{\mathcal{G}(k,m):\{k,m\} \in \textit{move}(t,1)\}) \nonumber \\
& = \textit{mex}(\{\mathcal{G}(k,1):\{k,1\} \in \textit{move}(t,1)\}\cup \{ \mathcal{G}(t,0)  \}) \nonumber \\
 & = \textit{mex}(\{\mathcal{G}(t-1,1), \cdots, \mathcal{G}(t-\lceil \frac{t}{2} \rceil,1)\} \cup \{ \mathcal{G}(t,0)  \}) \nonumber
\end{align}
Hence, from Equation (\ref{gtoeq0}), relation (\ref{no1in}), and the definition of the Grundy number in Definition \ref{defofmexgrundy}, we have $\mathcal{G}(t,1) = 1$. \\
$(iii)$
Suppose that 
\begin{equation}
\mathcal{G}(t,0) = 2.\label{grundye2}
\end{equation}
Then,
\begin{align}
& 2 \notin \{\mathcal{G}(k,0):\{k,0\} \in \textit{move}(t,0)\} \nonumber \\
& = \{\mathcal{G}(t-1,0), \cdots, \mathcal{G}(t- \lceil \frac{t}{2}  \rceil ,0)\}. \label{no2in}
\end{align}
Therefore, according to relation (\ref{kinn}) and the mathematical induction hypothesis for $(ii)$, $(iii)$, $(iv)$, and $(v)$, we have
\begin{equation}
0 \notin \{\mathcal{G}(t-1,1), \cdots, \mathcal{G}(t- \lceil \frac{t}{2}  \rceil,1)\}. \label{not0in2}
\end{equation}
\begin{align}
\mathcal{G}(t,1) & =\textit{mex}(\{\mathcal{G}(k,m):\{k,m\} \in \textit{move}(t,1)\}) \nonumber \\
& = \textit{mex}(\{ \mathcal{G}(t-1,1), \cdots, \mathcal{G}(t- \lceil \frac{t}{2}  \rceil,1)\}\cup \{ \mathcal{G}(t,0)  \}), \nonumber 
\end{align}
Hence, from Equation (\ref{grundye2}) and relation (\ref{not0in2}), we have $\mathcal{G}(t,1) = 0$. \\
$(iv)$ Suppose that 
\begin{equation}
\mathcal{G}(t,0)=2m \label{gt0eq2m}   
\end{equation}
 for a natural number $m$ such that $m > 1$. We proved that $\mathcal{G}(t,1) = 2m-1$. Because
\begin{equation}
2m = \mathcal{G}(t,0) =\textit{mex}(\{\mathcal{G}(k,0):\{k,0\} \in \textit{move}(t,0)\}), \nonumber
\end{equation}
\begin{equation}
\{\mathcal{G}(k,0):\{k,0\} \in \textit{move}(t,0)\} \supset \{2m-1, 2m-2,...,4,3,2,1,0\}\label{include2mminus1}
\end{equation}
 and
\begin{equation} 
2m \notin \{\mathcal{G}(k,0):\{k,0\} \in \textit{move}(t,0)\}.\label{not2min}
 \end{equation}
From the relations (\ref{kinn}), (\ref{include2mminus1}), and (\ref{not2min}) and the mathematical induction hypothesis for (ii),(iii), (iv), and (v), 
\begin{equation}
\{\mathcal{G}(k,1):\{k,1\} \in \textit{move}(t,1)\} \supset \{2m, 2m-3,2m-2,2m-5,2m-4,...,3,4,0,2,1\} \nonumber  
\end{equation}
and 
\begin{equation}
2m-1 \notin \{\mathcal{G}(k,m):\{k,1\} \in \textit{move}(t,1)\}. \nonumber  
\end{equation}
\begin{equation}
\mathcal{G}(t,1) =\textit{mex}(\{\mathcal{G}(k,1):\{k,1\} \in \textit{move}(t,1)\} \cup \{\mathcal{G}(t,0)\}); \nonumber  
\end{equation}
hence, according to (\ref{gt0eq2m}), $\mathcal{G}(t,1) = 2m-1 $. \\
$(v)$ Suppose that
\begin{equation}
\mathcal{G}(t,0)=2m-1 \label{gt02mminus1}
\end{equation}
 for natural number $m$.
Since 
\begin{equation}
2m-1 = \mathcal{G}(t,0) =\textit{mex}(\{\mathcal{G}(k,0):\{k,0\} \in \textit{move}(t,0)\}), \nonumber
\end{equation}
\begin{equation}
\{\mathcal{G}(k,0):\{k,0\} \in \textit{move}(t,0)\} \supset \{ 2m-2,2m-3...,2,1,0\} \label{inc2mminus2}
\end{equation}
and 
\begin{equation}
2m-1 \notin \{\mathcal{G}(k,0):\{k,0\} \in \textit{move}(t,0)\}. \label{not2mminus1}
\end{equation}
From relations (\ref{kinn}), (\ref{inc2mminus2}), and (\ref{not2mminus1}) and the mathematical induction hypothesis for $(ii)$, $(iii)$, $(iv)$, and $(v)$, 
\begin{equation}
\{\mathcal{G}(k,1):\{k,1\} \in \textit{move}(t,1)\}  \supset \{2m-3,2m-2,2m-5,2m-4,...0,2,1\}\label{inc2mminus3}
\end{equation}
and
\begin{equation}
2m \notin \{\mathcal{G}(k,1):\{k,1\} \in \textit{move}(t,1)\}.\label{notin2mm}
\end{equation}
As 
\begin{equation}
\mathcal{G}(t,1) =\textit{mex}(\{\mathcal{G}(k,1):\{k,1\} \in \textit{move}(t,1)\} \cup \{\mathcal{G}(t,0)\}),\nonumber
\end{equation}
according to Equation (\ref{gt02mminus1}), relations (\ref{inc2mminus3}) and (\ref{notin2mm}), we have 
$\mathcal{G}(t,1) = 2m $.
\end{proof}
	
\subsection{Three-Pile Maximum Nim with a Pass}\label{twowithpass}
Here, we study maximum Nim with three piles based on Definition \ref{defofmaxnim3piles} by modifying the standard rules of the games to allow a one-time pass. We consider only three-pile games, although generalization to the case of an arbitrary natural number of games is straightforward.

We denote the position of the game with three coordinates $\{s,t,u,p\}$, where $s$, $t$, and $u$ represent the numbers of stones in the first, second, and third piles, respectively. $p = 1$ if the pass is still available, and $p = 0$ otherwise.

We define a $\textit{move}$ in this game as follows.
\begin{defn}\label{defof3pilepass}
 For any $s,t,u \in Z_{\ge 0}$, we have $(i)$ and $(ii)$. \\
 $(i)$ If $p=1$ and $s+t+u>0$,
\begin{align}
& \textit{move}(s,t,u,p)= \{\{s-v,t,u,p\}: v \leq \lceil \frac{s}{2}  \rceil \text{ and } v \in N \}\nonumber \\
& \cup \{\{s,t-v,u,p\}:v \leq \lceil \frac{t}{2}  \rceil \text{ and } v \in N \}  \nonumber \\
& \cup \{\{s,t,u-v,p\}:v \leq \lceil \frac{u}{2}  \rceil \text{ and } v \in N \} \cup \{\{s,t,u,0\}\}. \nonumber
\end{align}
$(ii)$ If $p=0$ or $s+t+u=0$,
 \begin{align}
& \textit{move}(s,t,u,p)= \{\{s-v,t,u,p\}:v \leq \lceil \frac{s}{2}  \rceil \text{ and } v \in N \}\nonumber \\
& \cup \{\{s,t-v,u,p\}:v \leq \lceil \frac{t}{2}  \rceil \text{ and } v \in N \}  \nonumber \\
& \cup \{\{s,t,u-v,p\}:v \leq \lceil \frac{u}{2}  \rceil \text{ and } v \in N \} . \nonumber
\end{align}
\end{defn}

According to Definitions \ref{defofmexgrundy} and \ref{defof3pilepass}, we define the Grundy number $\mathcal{G}(s,t,u,p)$ of the position $\{s,t,u,p\}$.

\begin{rem}
Note that a pass is not available from the position $\{s,t,u,1\}$ with $s+t+u=0$ which is the terminal position. 
It is clear that $\mathcal{G}(s,0,0,p)$, $\mathcal{G}(0,s,0,p)$, and $\mathcal{G}(0,0,s,p)$ are identical to $\mathcal{G}(s,p)$ in Section \ref{sectionforceilpass}.
\end{rem}

\begin{lemma}\label{caseof1110}
Let $\mathcal{G}(s,t,u,p)$ be the Grundy number of position $\{s,t,u,p\}$. Then, we obtain the following equations: \\
$(i)$   $\mathcal{G}(1,0,0,0) = \mathcal{G}(0,1,0,0)=\mathcal{G}(0,0,1,0)=1.$ \\
$(ii)$  $\mathcal{G}(1,1,0,0) = \mathcal{G}(0,1,1,0)=\mathcal{G}(1,0,1,0)=0.$ \\
$(iii)$  $\mathcal{G}(1,1,1,0) =1.$  \\
$(iv)$   $\mathcal{G}(1,0,0,1) = \mathcal{G}(0,1,0,1)=\mathcal{G}(0,0,1,1)=2.$ \\
$(v)$   $\mathcal{G}(1,1,0,1) = \mathcal{G}(0,1,1,1)=\mathcal{G}(1,0,1,1)=1.$  \\
$(vi)$   $\mathcal{G}(1,1,1,1) =0. $
\end{lemma}
\begin{proof}
$(i)$  From Lemma \ref{grundyhalfnim}, $ \mathcal{G}(1,0,0,0) = \mathcal{G}(0,1,0,0)=\mathcal{G}(0,0,1,0)$  $=\mathcal{G}(1,0)$  \\ $=\mathcal{G}(1)=1.$\\
$(ii)$ From (i) and Theorem \ref{thofsumofgame}, we have
$\mathcal{G}(1,1,0,0)=\mathcal{G}(0,1,1,0)=\mathcal{G}(1,0,1,0)$
$=\mathcal{G}(1,0,0,0) \oplus \mathcal{G}(0,0,1,0)=1 \oplus 1 =0.$\\
$(iii)$ From (i) and Theorem \ref{thofsumofgame}, $\mathcal{G}(1,1,1,0) =\mathcal{G}(1,0,0,0) \oplus \mathcal{G}(0,1,0,0) \oplus \mathcal{G}(0,0,1,0)$ $= 1.$\\
$(iv)$ From (i) of Lemma \ref{grundyhalfnim}  and (v) of Theorem \ref{grundyhalfnimpass}, 
$\mathcal{G}(1,0,0,1) = \mathcal{G}(0,1,0,1)$ \\
$=\mathcal{G}(0,0,1,1)$ $=\mathcal{G}(1,1)=\mathcal{G}(1,0)+1=2.$ \\
$(v)$ From (ii) and (iv),
 \begin{align}
& \mathcal{G}(1,1,0,1) =\textit{mex}(\{\mathcal{G}(h,k,0,1):\{h,k,0,1\} \in \textit{move}(1,1,0,1)\} \cup \{\mathcal{G}(1,1,0,0)\})  \nonumber \\
& = \textit{mex}(\{\mathcal{G}(1,0,0,1), \mathcal{G}(0,1,0,1), \mathcal{G}(1,1,0,0)\})= \textit{mex}(\{2,2,0 \}) = 1. \nonumber
\end{align}
$(vi)$ From (iii) and (v),
 \begin{align}
& \mathcal{G}(1,1,1,1) =\textit{mex}(\{\mathcal{G}(h,k,j,1):\{h,k,j,1\} \in \textit{move}(1,1,1,1)\} \cup \{\mathcal{G}(1,1,1,0)\}) \nonumber \\
& = \textit{mex}(\{\mathcal{G}(1,1,0,1), \mathcal{G}(1,0,1,1), \mathcal{G}(0,1,1,1), \mathcal{G}(1,1,1,0)\})= \textit{mex}(\{1,1,1,1 \}) = 0. \nonumber
\end{align}
\end{proof}

\begin{theorem}\label{grundytwopilepass}
We have the following formulae for the Grundy numbers: \\
$(i)$  $\mathcal{G}(s,0,0,1) =\mathcal{G}(0,s,0,1) =\mathcal{G}(0,0,s,1) = \mathcal{G}(s,1)$ for $s \in Z_{\ge 0}$. \\
$(ii)$ We suppose that $s,t>0$, $t,u>0$, or $u,s>0$. Thus, we have the following: \ \
$(ii.1)$ For any $m \in Z_{\ge 0}$, if $ \mathcal{G}(s,t,u,0) = 2m$, then $ \mathcal{G}(s,t,u,1) = 2m+1.$\\
$(ii.2)$ For any $m \in Z_{\ge 0}$, if $ \mathcal{G}(s,t,u,0) = 2m+1$, then $ \mathcal{G}(s,t,u,1) = 2m.$
\end{theorem}
\begin{proof}
(I) Let $\{s,t,u,1\}$ be the position of the game.
The position $\{s,0,0,1\}$ is identical to the position $\{s,1\}$ in the game of Subsection \ref{sectionforceil}; hence, $\mathcal{G}(s,0,0,1) = \mathcal{G}(s,1)$.
Similarly, $\mathcal{G}(0,s,0,1) =\mathcal{G}(0,0,s,1) = \mathcal{G}(s,1)$, and we have $(i)$. \\
(II) We prove this using mathematical induction. From Lemma \ref{caseof1110}, $(ii.1)$ and $(ii.2)$ are valid for $s,t,u \in Z_{\ge 0}$ such that $s,t,u \leq 1$.  

Suppose that $(ii.1)$ and $(ii.2)$ for 
$\{h,k,j,p\}$ when $h \leq s, k \leq t$, $j \leq u$, $h+k+j<s+t+u$, and $p=0,1$.\\
(II.1)	Suppose that
\begin{equation}
 \mathcal{G}(s,t,u,0) =2m \label{geq2m}
\end{equation}
for $m \in Z_{\ge 0}$. Then, according to the definition of the Grundy number in Definition \ref{defofmexgrundy},
\begin{align}
& \{\mathcal{G}(h,k,j,0):\{h,k,j,0\} \in \textit{move}(s,t,u,0)\}\nonumber \\
\supset & \{2m-1, 2m-2,2m-3,2m-4, \cdots, 5,4,3,2,1,0\}\label{movefor0}
\end{align}
and
\begin{equation}
2m \notin \{\mathcal{G}(h,k,j,0):\{h,k,j,0\} \in \textit{move}(s,t,u,0)\}.\label{no2m}
\end{equation}
\underline{Case $(a)$} 
Suppose that
\begin{equation}
 \textit{move}(s,t,u,0) \cap \{\{s,0,0,0\},\{0,t,0,0\},\{0,0,u,0\}\} = \emptyset.\nonumber 
\end{equation}
Then, we have 
\begin{equation}
h,k>0 \text{ or } k,j>0 \text{ or } j,h>0 \label{condforhkj}
\end{equation}
when 
\begin{equation}
\{h,k,j,0\} \in  \textit{move}(s,t,u,0). \nonumber 
\end{equation}
By applying the mathematical induction hypothesis to $(ii.1)$ and $(ii.2)$, along with relations (\ref{movefor0}), (\ref{no2m}), and (\ref{condforhkj}), we obtain
\begin{align}
& \{\mathcal{G}(h,k,j,1):\{h,k,j,1\} \in \textit{move}(s,t,u,1)\} \nonumber \\
& \supset \{2m-2, 2m-1,2m-4,2m-3, \cdots, 4,5,2,3,0,1\} \label{incl2mninus2b}
\end{align}
and
\begin{equation}
 2m+1 \notin \{\mathcal{G}(h,k,j,1):\{h,k,j,1\} \in \textit{move}(s,t,u,1)\}. \label{no2mplus1b}
\end{equation}
As
\begin{equation}
\mathcal{G}(s,t,u,1) =\textit{mex}(\{\mathcal{G}(h,k,j,1):\{h,k,j,1\} \in \textit{move}(s,t,u,1)\} \cup \{\mathcal{G}(s,t,u,0)\}), \nonumber
\end{equation}
according to Equations (\ref{geq2m}), (\ref{incl2mninus2b}), and (\ref{no2mplus1b}), we have 
$\mathcal{G}(s,t,u,1)=2m+1$.\\
\underline{Case $(b)$} 
Suppose that 
\begin{align}
&  \textit{move}(s,t,u,0) \cap \{\{s,0,0,0\},\{0,t,0,0\},\{0,0,u,0\}\} \ne \emptyset. \nonumber 
\end{align}
Then, we have 
$\{t,u\}=\{1,0\}$ or $\{t,u\}=\{0,1\}$ or $\{s,u\}=\{1,0\}$ or $\{s,u\}=\{0,1\}$ or $\{s,t\}=\{1,0\}$ or $\{s,t\}=\{0,1\}$.

Here, we prove only the case where $\{t,u\}=\{1,0\}$. If $s = 1$, according to Lemma \ref{caseof1110}, we have 
$\mathcal{G}(1,1,0,0) = 0$ and $\mathcal{G}(1,1,0,1) = 1=\mathcal{G}(1,1,0,0)+1$. This result satisfies $(ii.1)$.

Therefore, we assume the following:
\begin{equation}\label{gs002m}
\mathcal{G}(s,1,0,0) = 2m,
\end{equation}
where $s>1$, and we prove that $\mathcal{G}(s,1,0,1) = 2m+1$.
Because $s>1$, we have
\begin{equation}
  \textit{move}(s,1,0,0) = \{\{s-v,1,0,0\}: v \in N \text{ and } v \leq \lceil \frac{s}{2} \rceil\}\cup \{\{s,0,0,0\}\}, \label{nowhere0100}
 \end{equation} 
where 
 \begin{equation} 
s-v > s - \lceil \frac{s}{2}  \rceil >0 \label{nowhere0100b}
\end{equation}
 for $v$ such that $v \leq \lceil \frac{s}{2}  \rceil$.

\begin{equation}\label{s002m}
\mathcal{G}(s,1,0,0)=\mathcal{G}(s,0,0,0)  \oplus \mathcal{G}(0,1,0,0) = \mathcal{G}(s,0,0,0)  \oplus \mathcal{G}(1,0)= \mathcal{G}(s,0,0,0)  \oplus 1\nonumber
\end{equation}
Hence, according to Equation (\ref{gs002m}), we have
\begin{equation}\label{s0002m1}
\mathcal{G}(s,0,0,0)=2m+1.
\end{equation}
We use relation (\ref{movefor0}) for $\{s,t,u,0\} = \{s,1,0,0\}$; then, we have
\begin{align}
& \{\mathcal{G}(h,k,0,0):\{h,k,0,0\} \in \textit{move}(s,1,0,0)\} \nonumber \\
& = \{\mathcal{G}(h,1,0,0):\{h,1,0,0\} \in \textit{move}(s,1,0,0)\} \cup \{\mathcal{G}(s,0,0,0)\} \nonumber \\ 
& \supset  \{2m-1,2m-2,2m-3,2m-4, \cdots,5,4,3,2,1,0\};\nonumber
\end{align}
hence, from Equation (\ref{s0002m1}), 
we have 
\begin{align}
& \{\mathcal{G}(h,1,0,0):\{h,1,0,0\} \in \textit{move}(s,1,0,0)\}\nonumber \\
& = \{\{s-v,1,0,0\}: v \in N \text{ and } v \leq \lceil \frac{s}{2} \rceil\} \nonumber \\
& \supset  \{2m-1,2m-2,2m-3,2m-4, \cdots,5,4,3,2,1,0\}.\label{movefor02}
\end{align}
We use relation (\ref{no2m}) for $\{s,t,u,0\} = \{s,1,0,0\}$; then, we have 
\begin{equation}
2m \notin \{\mathcal{G}(h,1,0,0):\{h,1,0,0\} \in \textit{move}(s,1,0,0)\}\cup \{\{s,0,0,0\}\}.\label{no2m2}
\end{equation}
From equations (\ref{nowhere0100}) and (\ref{movefor02}), the inequality in (\ref{nowhere0100b}), and the mathematical induction hypothesis for $(ii.1)$ and $(ii.2)$, we have
\begin{align}
& \{\mathcal{G}(h,1,0,1):\{h,1,0,1\} \in \textit{move}(s,1,0,1)\} \nonumber \\
& \supset \{2m-2, 2m-1,2m-4,2m-3, \cdots, 4,5,2,3,0,1\}. \label{ind2mminus2c}
\end{align}
From equations (\ref{nowhere0100}) and (\ref{no2m2}), the inequality in (\ref{nowhere0100b}), and the mathematical induction hypothesis for $(ii.1)$ and $(ii.2)$, we have 
\begin{equation}
 2m+1 \notin \{\mathcal{G}(h,1,0,1):\{h,1,0,1\} \in \textit{move}(s,1,0,1)\}. \label{notincluding2m1}
\end{equation}
From Equations $(\ref{s0002m1})$ and $(iv)$ of Theorem \ref{grundyhalfnimpass}, we have
\begin{equation}
\mathcal{G}(s,0,0,1) =\mathcal{G}(s,1)=\mathcal{G}(s,0)+1= \mathcal{G}(s,0,0,0)+1 = 2m+2.\label{s0002m2}
\end{equation}
Since 
\begin{align}
& \mathcal{G}(s,1,0,1) =\textit{mex}(\{\mathcal{G}(h,k,0,1):\{h,k,0,1\} \in \textit{move}(s,1,0,1)\} \cup \{\mathcal{G}(s,1,0,0)\})  \nonumber \\
& = \textit{mex}(\{\mathcal{G}(h,1,0,1):\{h,1,0,1\} \in \textit{move}(s,1,0,1)\}  \nonumber \\
& \cup \{\mathcal{G}(s,0,0,1)\} \cup \{\mathcal{G}(s,1,0,0)\}), \nonumber 
\end{align}  
by the relations (\ref{ind2mminus2c}), (\ref{notincluding2m1}), and equations (\ref{gs002m}) and (\ref{s0002m2}), we have 
$\mathcal{G}(s,1,0,1)$ \\
$=2m+1$.\\
(II.2) We assume that: 
\begin{equation}
 \mathcal{G}(s,t,u,0) =2m+1 \label{geq2mb}
\end{equation}
for $m \in Z_{\ge 0}$. Then, according to the definition of the Grundy number in Definition \ref{defofmexgrundy}
\begin{align}
& \{\mathcal{G}(h,k,j,0):\{h,k,j,0\} \in \textit{move}(s,t,u,0)\}\nonumber \\
\supset & \{2m,2m-1, 2m-2,2m-3, \cdots, 5,4,3,2,1,0\}\label{movefor0b}
\end{align}
and
\begin{equation}
2m+1 \notin \{\mathcal{G}(h,k,j,0):\{h,k,j,0\} \in \textit{move}(s,t,u,0)\}.\label{no2mb}
\end{equation}
\underline{Case $(a)$} 
Suppose that
\begin{align}
&  \textit{move}(s,t,u,0) \cap \{\{s,0,0,0\},\{0,t,0,0\},\{0,0,u,0\}\} = \emptyset. \nonumber 
\end{align}
Then, we have 
\begin{equation}
h,k>0 \text{ or } k,j>0 \text{ or } j,h>0 \label{condforhkj2}
\end{equation}
when 
\begin{equation}
\{h,k,j,0\} \in  \textit{move}(s,t,u,0). \nonumber 
\end{equation}
According to the inequality in (\ref{condforhkj2}), relations (\ref{movefor0b}) and (\ref{no2mb}), and the mathematical induction hypothesis for $(ii.1)$ and $(ii.2)$, 
we have 
\begin{align}
& \{\mathcal{G}(h,k,j,1):\{h,k,j,1\} \in \textit{move}(s,t,u,1)\} \nonumber \\
& \supset \{2m+1,2m-2, 2m-1,2m-4,2m-3, \cdots, 4,5,2,3,0,1\} \label{incc2mplus1}
\end{align}
and 
\begin{equation}
 2m \notin \{\mathcal{G}(h,k,j,1):\{h,k,j,1\} \in \textit{move}(s,t,u,1)\}. \label{no2mint}
\end{equation}
As 
\begin{equation}
\mathcal{G}(s,t,u,1) =\textit{mex}(\{\mathcal{G}(h,k,j,1):\{h,k,j,1\} \in \textit{move}(s,t,u,1)\} \cup \{\mathcal{G}(s,t,u,0)\}), \nonumber
\end{equation}
 according to Equations (\ref{geq2mb}), (\ref{incc2mplus1}), and (\ref{no2mint}), we have 
$\mathcal{G}(s,t,u,1)=2m$. \\
\underline{Case $(b)$} 
We suppose that 
\begin{align}
&  \textit{move}(s,t,u,0) \cap \{\{s,0,0,0\},\{0,t,0,0\},\{0,0,u,0\}\} \ne \emptyset. \nonumber 
\end{align}
Then, we have 
$\{t,u\}=\{1,0\}$ or $\{t,u\}=\{0,1\}$ or $\{s,u\}=\{1,0\}$ or $\{s,u\}=\{0,1\}$ or $\{s,t\}=\{1,0\}$ or $\{s,t\}=\{0,1\}$.

Here, we prove only the case where $\{t,u\}=\{1,0\}$. If $s = 1$, according to Lemma \ref{caseof1110}, we have 
$\mathcal{G}(1,1,0,0) = 0$. This contradicts Equation  (\ref{geq2mb}). 

We suppose that 
\begin{equation}\label{s10002mbb}
\mathcal{G}(s,1,0,0) = 2m+1
\end{equation}
and $s>1$. 

Because $s>1$, we have
\begin{equation}
  \textit{move}(s,1,0,0) = \{\{s-v,1,0,0\}: v \in N \text{ and } v \leq \lceil \frac{s}{2} \rceil\}\cup \{\{s,0,0,0\}\}, \label{nowhere0100bb}
 \end{equation} 
where 
 \begin{equation} 
s-v > s - \lceil \frac{s}{2}  \rceil >0 \label{nowhere0100bbb}
\end{equation}
 for $v$ such that $v \leq \lceil \frac{s}{2}  \rceil$.
\begin{equation}\label{s002mb}
\mathcal{G}(s,1,0,0)=\mathcal{G}(s,0,0,0)  \oplus \mathcal{G}(0,1,0,0) = \mathcal{G}(s,0,0,0)  \oplus \mathcal{G}(1,0)= \mathcal{G}(s,0,0,0)  \oplus 1 \nonumber
\end{equation}
Hence, according to Equation (\ref{s10002mbb}), we have
\begin{equation}\label{s0002m1b}
\mathcal{G}(s,0,0,0)=2m.
\end{equation}
From relation (\ref{movefor0b}), 
\begin{align}
& \{\mathcal{G}(h,1,0,0):\{h,k,0,0\} \in \textit{move}(s,1,0,0)\}\nonumber \\
& =\{\mathcal{G}(h,1,0,0):\{h,1,0,0\} \in \textit{move}(s,1,0,0)\}\cup \{\mathcal{G}(s,0,0,0) \}\nonumber \\
\supset & \{2m,2m-1, 2m-2,2m-3, \cdots, 5,4,3,2,1,0\};\label{movefor0b2a}
\end{align}
hence, according to Equation (\ref{s0002m1b}),
we have 
\begin{align}
& \{\mathcal{G}(h,1,0,0):\{h,1,0,0\} \in \textit{move}(s,1,0,0)\}\nonumber \\
\supset & \{2m-1, 2m-2,2m-3, \cdots, 5,4,3,2,1,0\}.\label{movefor0b2}
\end{align}
From relations (\ref{nowhere0100bb}) and (\ref{movefor0b2}), the inequality in (\ref{nowhere0100bbb}), and the mathematical induction hypothesis for $(ii.1)$ and $(ii.2)$, 
we have
\begin{align}
& \{\mathcal{G}(h,1,0,1):\{h,1,0,1\} \in \textit{move}(s,1,0,1)\} \nonumber \\
& \supset \{2m-2, 2m-1,2m-4,2m-3, \cdots, 4,5,2,3,0,1\}. \label{include2mm201}
\end{align}
From relation (\ref{no2mb}), we have 
\begin{align}
& 2m+1 \notin \{\mathcal{G}(h,k,0,0):\{h,k,0,0\} \in \textit{move}(s,1,0,0)\}. \nonumber \\
& = \{\mathcal{G}(h,1,0,0):\{h,1,0,0\} \in \textit{move}(s,1,0,0)\}\cup \{\mathcal{G}(s,0,0,0) \}.\label{no2mb2}
\end{align}
From relations (\ref{nowhere0100bb}), (\ref{no2mb2}), inequality in (\ref{nowhere0100bbb}), and the mathematical induction hypothesis on $(ii.1)$ and $(ii.2)$, we have 
\begin{equation}
 2m \notin \{\mathcal{G}(h,1,0,1):\{h,1,0,1\} \in \textit{move}(s,1,0,1)\}. \label{no2mintb}
\end{equation}
If $m=0$, according to Equations (\ref{s0002m1b}) and (ii) of Theorem \ref{grundyhalfnimpass}, we have
\begin{equation}\label{caseof1}
\mathcal{G}(s,0,0,1) = \mathcal{G}(s,1) =1.  
\end{equation}
If $m=1$, according to Equations (\ref{s0002m1b}) and (iii) of Theorem \ref{grundyhalfnimpass}, we have
\begin{equation}\label{caseof2}
\mathcal{G}(s,0,0,1) = \mathcal{G}(s,1) =0.  
\end{equation}
If $m \ne 0,1$, according to Equations (\ref{s0002m1b}) and (iv) of Theorem \ref{grundyhalfnimpass}, we have
\begin{equation}\label{caseof3}
\mathcal{G}(s,0,0,1) = \mathcal{G}(s,1)=2m-1.  
\end{equation}
From Equations (\ref{caseof1}), (\ref{caseof2}), and (\ref{caseof3}),
we have
\begin{equation}\label{caseof4}
\mathcal{G}(s,0,0,1) \ne 2m.  
\end{equation}
As
\begin{align}
& \mathcal{G}(s,1,0,1) =\textit{mex}(\{\mathcal{G}(h,1,0,1):\{h,1,0,1\} \in \textit{move}(s,1,0,1)\} \nonumber \\
& \cup \{\mathcal{G}(s,0,0,1)\}\cup \{\mathcal{G}(s,1,0,0)\}), \nonumber 
\end{align}
according to Equation (\ref{s10002mbb}), relations (\ref{include2mm201}) and (\ref{no2mintb}), and the inequality in (\ref{caseof4}), we have \\
$\mathcal{G}(s,t,u,1)=2m$.
\end{proof}

\begin{flushleft}
	\large{Acknowledgements}\\
	\normalsize{We would like to thank Editage (www.editage.com) for English language editing.}
\end{flushleft}


\begin{thebibliography}{111}
\bibitem{nimpass} R.E. Morrison, E.J. Friedman, and A.S. Landsberg, Combinatorial games with a pass: A
dynamic systems approach, Chaos, {\it An Interdisciplinary Journal of Nonlinear Science}, {\bf 21}
(2011), 43-108.
\bibitem{integers1} M. Inoue, M. Fukui, and R. Miyadera, IMPARTIAL CHOCOLATE BAR GAMES WITH A PASS, Integers Volume 16, 2016.
\bibitem{jcdcg2018a}R. Miyadera, S. Kannan, and M. Fukui, Some Formulas for Max Nim, {\it The 21th Japan Conference on Discrete and Computational Geometry, Graphs, and Games}, (2018) 35-36. 
\bibitem{lesson}M. H. Albert, R. J. Nowakowski, and D. Wolfe, {\it Lessons In Play}, A K Peters, p-139.
\bibitem{combysiegel} A. N. Siegel, {\it Combinatorial Game Theory }, Graduate Studies in Mathematics, American Mathematical Society (2013).
\bibitem{levinenim} L. Levine, Fractal sequences and restricted Nim, {\it Ars Combinatoria},{\bf 80} (2006) 113-127.

\end{thebibliography}
\end{document}